\documentclass[11pt,a4paper]{article}

\usepackage[T1]{fontenc}
\usepackage[utf8]{inputenc}

\usepackage{amsmath}
\usepackage[margin=1.1in]{geometry}
\usepackage{lmodern}
\usepackage[english]{babel}
\usepackage{graphicx,mathrsfs}
\usepackage{cases}
\usepackage{upgreek}
\usepackage{amsfonts}
\usepackage{amssymb} 
\usepackage{amsthm}
\usepackage{nicefrac}
\usepackage{bbm}
\usepackage{bm} 
\usepackage{lipsum}
\usepackage[normalem]{ulem}
\usepackage{subcaption}
\usepackage{xcolor}
\usepackage{url}
\usepackage{csquotes}
\usepackage{stmaryrd}
\usepackage{pgf,tikz}
\usepackage{pgfplots,pgfplotstable}
\usepackage{multirow}
\usepackage{wrapfig}
\usepackage{graphicx}
\graphicspath{ {./images/} }
\usepackage{subcaption}
\usepackage[ruled,vlined]{algorithm2e}
\usepackage{xcolor}

\allowdisplaybreaks
\numberwithin{equation}{section}

\theoremstyle{plain}
\newtheorem{theorem}{Theorem}[section]
\newtheorem{lemma}[theorem]{Lemma}

\theoremstyle{definition}

\def \be {\begin{equation}}
\def \ee {\end{equation}}

\def \E {\mathbb{E}}

\def\E{{\mathbb E}}

\def\bR{{\mathbb R}}
\def\bS{{\mathbb S}}

\def\bP{{\mathbb P}}

\def\cC{{\mathcal C}}

\renewcommand{\phi}{\varphi}
\renewcommand{\epsilon}{\varepsilon}
\renewcommand{\tilde}{\widetilde}
\renewcommand{\hat}{\widehat}

 \usepackage[doi=false,isbn=false,url=false,eprint=false, style=numeric,maxbibnames=99,giveninits=true,maxcitenames=99,backend=bibtex]{biblatex}
 \AtBeginBibliography{\footnotesize}
 
 \renewrobustcmd*{\bibinitdelim}{\,} 
 \RequirePackage[colorlinks,citecolor=blue,urlcolor=blue]{hyperref}

 \title{
 	{\textbf{It\^o's Formula for the \\ Rearranged Stochastic Heat Equation}}\\
 }
 \author{François Delarue and William R.P. Hammersley \thanks{F. Delarue and W. Hammersley are supported by the European Research Council (ERC) under the European Union’s Horizon 2020 research and innovation programme (ELISA project, Grant agreement No. 101054746). Email: francois.delarue@univ-cotedazur.fr, william.hammersley@univ-cotedazur.fr} \\ Université Côte d'Azur, CNRS, Laboratoire J.A. Dieudonné}
 
 \date{\today}

\begin{document}

 	\maketitle
 
 \begin{abstract}
 	The purpose of this 
 	short note is to 
 	prove a convenient version of It\^o's formula for 
 	the Rearranged Stochastic Heat Equation (RSHE) introduced by the two authors in a previous 
 	contribution. 
 	This equation is a penalised version of the standard Stochastic Heat Equation (SHE) on the circle subject to a coloured noise, 
 	whose solution is constrained to stay within the set of symmetric quantile functions by means of a reflection term. 
 	Here, we identity the generator of the solution
 	when it is acting 
 	on functions defined on the space ${\mathcal P}_2({\mathbb R})$ (of one-dimensional probability measures 
 	with a finite second moment) 
 	that are assumed to be smooth in Lions' sense. In particular, we prove that the reflection term 
 	in the RSHE is orthogonal to the Lions (or Wasserstein) 
 	derivative of smooth functions defined on ${\mathcal P}_2({\mathbb R})$. 
 	The proof relies on non-trivial bounds for the gradient of the solution to the RSHE.
 	\medskip
 	
 	\vspace{2pt} 

 \end{abstract}

\section{Introduction} 

\subsection{Rearranged Stochastic Heat Equation}
\label{subse:1.1}
The Rearranged Stochastic Heat Equation (RSHE) is a Stochastic Partial Differential Equation (SPDE) 
on $[0,\infty) \times {\mathbb S}$, where ${\mathbb S}={\mathbb R}/{\mathbb Z}$ is the standard circle 
of length 1, writing formally as 
\begin{equation}
\label{eq:RSHE} 
d X_t(x) = \Delta X_t(x) d t + d W_t(x) + d \eta_t(x), \quad x \in {\mathbb S}, \quad t \geq 0.  
\end{equation} 
Here $\Delta = D^2$ is the standard second-order derivative acting on 1-periodic 
functions and $(W_t)_{t \geq 0}$ is a coloured noise taking values in the space 
$L_{\rm sym}^2({\mathbb S})$ of  symmetric square-integrable  functions on ${\mathbb S}$ (w.r.t. 
the Lebesgue measure ${\rm Leb}_{\mathbb S}$ on ${\mathbb S}$) and expanding 
as
\begin{equation} 
\label{eq:expansion:W:beta:noises}
W_t(x) := \sum_{m \in {\mathbb N}_0} \lambda_m e_m(x) \beta_t^m, \quad t \geq 0,
\end{equation}  
where $((\beta^m_t)_{t \geq 0})_{m \in {\mathbb N}_0}$ is a collection of independent 
$1d$-Brownian motions constructed on a filtered probability space $(\Omega,{\mathcal F},{\mathbb F}=({\mathcal F}_t)_{t \geq 0},{\mathbb P})$
 (satisfying the usual conditions), $(e_0=1,(e_m = \sqrt{2} \cos (2 \pi m \cdot))_{m \in {\mathbb N}})$ is the standard 
symmetric Fourier basis on ${\mathbb S}$ and ${\mathbb N}_0:=\{0,1,2,\cdots\}$ is the collection of non-negative integers (including
$zero$). The collection $(\lambda_m)_{m \in {\mathbb N}_0}$ is a sequence of non-negative reals 
that is equivalent to $m^{-\lambda}$ for $m$ large, for a fixed $\lambda > 1/2$. 

The very novelty in 
\eqref{eq:RSHE} comes from the term $(\eta_t)_{t \geq 0}$ which is part of the solution itself and 
which forces the solution to stay within the space 
$U^2({\mathbb S})$
of 
so-called non-increasing 
elements of $L^2_{\rm sym}({\mathbb S})$, namely those elements with a (\textit{canonical}) representative 
that is symmetric with respect to $0$, non-increasing on $[0,1/2]$, right-continuous on 
$[0,1/2)$
and left-continuous at $1/2$. Intuitively, 
$U^2({\mathbb S})$ is a set of symmetric 
quantile functions and, when equipped with the $L^2$-norm $\| \cdot \|_2$ (on 
$L^2({\mathbb S})$), it can be  proven 
to be isometric
with ${\mathcal P}_2({\mathbb R})$ (the space 
of probability measures on ${\mathbb R}$ with a finite second moment) 
equipped with the $2$-Wasserstein distance
defined by 
$
{\mathcal W}_2(\mu,\nu) := \inf_{\pi} [ \int_{{\mathbb R}^2} \vert x-y \vert^2 d\pi(x,y) ]^{1/2}$,   for 
$\mu,\nu \in {\mathcal P}_2({\mathbb R})$, with 
the infimum being taken over the couplings $\pi$ of 
$\mu$ and $\nu$. 

Existence and uniqueness of a solution to 
\eqref{eq:RSHE} is addressed in 
\cite{delarueHammersley2022rshe}. 
Briefly, given an ${\mathcal F}_0$-measurable and $U^2({\mathbb S})$-valued initial condition $X_0$ 
satisfying ${\mathbb E}[ \| X_0 \|^p_2] < \infty$ for all $p \geq 1$, there exists a unique pair 
$(X_t,\eta_t)_{t \geq 0}$ such that 
\vskip 2pt

\noindent \textbf{RSHE.1} $(X_t)_{t \geq 0}$ is a continuous ${\mathbb F}$-adapted process with values in $
U^2({\mathbb S})$;
\vskip 1pt

\noindent  \textbf{RSHE.2} $(\eta_t)_{t \geq 0}$ is a continuous ${\mathbb F}$-adapted process with values in $H^{-2}_{\rm sym}({\mathbb S})$, starting from $0$ at $0$, such that, with probability 1, for any $u \in H^2_{\rm sym}({\mathbb S})$ that is non-increasing, the path
$(\langle \eta_t,u \rangle)_{t \geq 0}$ is non-decreasing;
\vskip 1pt

\noindent \textbf{RSHE.3}   ${\mathbb P}$-a.s., for any $u \in H^{2}_{\rm sym}({\mathbb S})$, any $t \geq 0$,  
$ \langle  X_t,u \rangle    =     \int_0^t \langle   {X}_r  ,\Delta     u \rangle dr +\langle W_t  ,u \rangle + \langle \eta_t ,u\rangle;$
 \vskip 1pt

\noindent  \textbf{RSHE.4}  for any $t \geq 0$, 
$\lim_{\varepsilon\searrow0}\mathbb{E}[ \int_0^t  e^{\varepsilon\Delta}X_r   \cdot d  \eta_r ]= 0.$
 \vskip 2pt
 
Above, 
$H^{\mu}_{\rm sym}({\mathbb S})$
(here with $\mu=-2,2$) denotes the Sobolev space of symmetric functions/distributions 
$f$ such that
$\| f \|_{2,\mu}^2 := \sum_{m \in {\mathbb N}_0} (m \vee 1)^{2 \mu}  \hat{f}_m^2 < \infty$,
 where,
for a distribution on the torus, 
$ \hat{f}_m$
is the $m^{\text{th}}$ (cosine) Fourier coefficient of $f$
(i.e., 
  $\hat{f}_m:= \int_{\bS} f(x) e_m(x) dx$ when $f$ is an integrable function  on ${\mathbb S}$).  
The corresponding inner product is denoted $\langle \cdot,\cdot \rangle_{2,\mu}$.
  
One of the main contribution of  
\cite{delarueHammersley2022rshe}
is to clarify the definition of the integral with respect to the process $(\eta_t)_{t \geq 0}$. 
For an ${\mathbb F}$-adapted process $(Z_t)_{t \geq 0}$, with continuous trajectories from $[0,\infty)$ to $L^2_{\rm sym}({\mathbb S})$,
the integral process $(\int_0^t e^{\varepsilon \Delta} Z_r \cdot d \eta_r)_{t \geq 0}$, where 
$(e^{t \Delta})_{t \geq 0}$ denotes the standard heat semi-group, can be constructed, path by path, 
as the limit of standard Riemann sums
$(\sum_{k \geq 0 : k h \leq t} e^{\varepsilon\Delta } Z_{kh} \cdot ( \eta_{(k+1)h} - \eta_{kh}))_{t \geq 0}$ as $h$ tends to $0$. 
The dot product in the Riemann sums is understood as a duality in $H^2_{\rm sym}({\mathbb S})/H^{-2}_{\rm sym}({\mathbb S})$. 
The integral is ${\mathbb F}$-adapted and continuous in time and, importantly, it retains the property \textbf{RSHE.2} and is non-decreasing when $(Z_t)_{t \geq 0}$ takes values in 
$U^2({\mathbb S})$. 
In particular, 
the latter says that the integral in 
\textbf{RSHE.4} 
is non-negative: the fact that the limit as $\varepsilon$ tends to $0$ suggests that 
 $\int_0^t  X_r   \cdot d  \eta_r =0$, which is consistent with standard orthogonality rules in reflected equations. Note that this claim is  informal 
 because $(X_r)_{r \geq 0}$ does not take values in $H^2_{\rm sym}({\mathbb S})$. A key feature is that $U^2({\mathbb S})$ is preserved by the heat semi-group.

\subsection{Main result}

The main purpose of this note is to prove a convenient form of It\^o expansion for the solution $(X_t)_{t \geq 0}$
and, as such, to identify the generator of the dynamics when tested on an appropriate class of test functions. 
Precisely, we are able to compute the action of the dynamics along smooth test functions
$\varphi : {\mathcal P}_2({\mathbb R}) \rightarrow {\mathbb R}$
(which we call `mean-field' functions). We refer to \cite{delarueOuknine2024intrinsicRegMFG} for an application. 

Differentiability is here understood in Lions' sense. 
We recall that a function $\varphi  : {\mathcal P}_2({\mathbb R}) \rightarrow {\mathbb R}$ is said to 
Lions continuously differentiable if the mapping $u \in L^2({\mathbb S}) \mapsto \varphi ({\rm Leb}_{\mathbb S} \circ u^{-1})$
is Fréchet continuously differentiable, in which case the Fréchet derivative at $u \in L^2({\mathbb S})$ can written 
$x \in {\mathbb S} \mapsto \partial_\mu \varphi({\rm Leb}_{\mathbb S} \circ u^{-1},u(x))$
for a function $y \in {\mathbb R} \mapsto \partial_\mu \varphi({\rm Leb}_{\mathbb S} \circ u^{-1},y) \in L^2({\mathbb R},{\rm Leb}_{\mathbb S} \circ u^{-1})$
depending only on the law ${\rm Leb}_{\mathbb S} \circ u^{-1}$
of $u$ (seen as a random variable on the torus). 
Below, we work with \textit{smooth} functions $\varphi$ for which 
the derivative $(\mu,y) \in {\mathcal P}_2({\mathbb R}) \times {\mathbb R} \mapsto 
\nabla_y \partial_\mu \varphi(\mu,y)$ is jointly continuous and can be further differentiated in $\mu$ and $y$. 
We refer to \cite[Chapter 5]{CarmonaDelarueI} for details.

\begin{theorem}
	\label{prop:ito:b}
For any 
	smooth mean-field function 
	$\varphi : {\mathcal P}({\mathbb R})\rightarrow \bR$
	with bounded and jointly continuous derivatives $\partial_\mu \varphi : {\mathcal P}({\mathbb R}) 
	\times {\mathbb R} \rightarrow {\mathbb R}$, $\nabla_y \partial_\mu \varphi : {\mathcal P}({\mathbb R}) 
	\times {\mathbb R} \rightarrow {\mathbb R}$
	and $\partial^2_\mu \varphi: {\mathcal P}({\mathbb R}) 
	\times {\mathbb R} \times {\mathbb R} \rightarrow {\mathbb R}$, with 
	${\mathcal P}({\mathbb S})$ being equipped with the $2$-Wasserstein distance, the following formula holds, with probability 1, for any $t \geq 0$, 
	\begin{equation}
\label{eq:main:ito}
		\begin{split}
			&\varphi(\mu_t) 
			= 
			\varphi(\mu_0)  - \int_0^t 
			\int_{\mathbb S} \nabla_y \partial_{\mu}
			\varphi(\mu_s)
			\bigl(X_s(x)\bigr) \bigl[ D X_s(x) \bigr]^2 dx \, ds
			\\
			&\hspace{15pt}  +
			\int_0^t 
			\int_{\mathbb S}  \partial_{\mu}
			\varphi(\mu_s) 
			\bigl(X_s(x)\bigr)  d W_s(x)
			 +
			\frac12
			\int_0^t  \int_{\mathbb S} \nabla_y \partial_\mu \varphi (\mu_s)\bigl( X_s(x) 
			\bigr) F_1(x) dx \, ds
			\\
			&\hspace{15pt} 
			+ \frac12  \int_0^t \int_{\mathbb S}\int_{\mathbb S} \partial^2_{\mu} \varphi (\mu_s) \bigl( X_s(x),X_s(y) 
			\bigr) F_2(x,y) dx \, dy \, ds,
		\end{split}
	\end{equation}
	where 
	$\mu_t:=\textrm{\rm Leb}_{\mathbb S} \circ X_t^{-1}$ 
	and with
	\begin{equation} 
	\label{eq:F1:F2}
		\begin{split}
			&F_1(x) := \sum_{k \in {\mathbb N}_0} \lambda_k^2 e_k^2(x), \quad 
			F_2(x,y) := \sum_{k \in {\mathbb N}_0}  \lambda_k^2 e_k(x)e_k(y), \quad 
			x,y \in {\mathbb S}. 
		\end{split}
	\end{equation}
\end{theorem}

The striking point is that the reflection term does not appear in the expansion. Somehow, this says that the 
Lions derivative $x \in {\mathbb S} \mapsto \partial_\mu \varphi(\textrm{\rm Leb}_{\mathbb S} \circ X_t^{-1})(X_t(x))$, computed
at the current value of the solution, is 
orthogonal to the instantaneous 
reflecting term 
$d \eta_t$ in 
\eqref{eq:RSHE}. 
Equivalently, 
this provides a clear way to identify the generator ${\mathscr L}$ of the process $(\mu_t)_{t \geq 0}$, which may be regarded 
as a diffusion process with values in the space ${\mathcal P}_2({\mathbb R})$. In clear,
for $\mu \in {\mathcal P}_2({\mathbb R})$
with a {symmetric} quantile function $X \in U^2({\mathbb S})$ (which is uniquely defined) in $H^1({\mathbb S})$, 
 \begin{equation*} 
 \begin{split}
 {\mathscr L} \varphi(\mu) &= 
 - \int_{\mathbb S} 
 \partial_y \partial_{\mu} \varphi(\mu)(X(x)) [\nabla X(x)]^2 dx 
 +
 \frac12 \int_{\mathbb S} 
\partial_y \partial_{\mu} \varphi(\mu) (X(x)) F_1(x) dx 
\\
&\hspace{15pt} +
 \frac12 \int_{\mathbb S} \int_{\mathbb S} 
 \partial^2_{\mu} \varphi(\mu) (X(x),X(y)) F_2(x,y) dx \, dy.  
 \end{split}
 \end{equation*} 
 This is part of our paper to prove that $(X_t)_{t \geq 0}$ is in fact $H^1({\mathbb S})$-valued in positive time, thus justifying the existence of 
 the derivative of $X$ in the definition of ${\mathscr L}$ (when $\mu$ is chosen as $\mu_t$). 

The formula can be easily extended to drifted versions of the RSHE, namely when a term 
$B(\mu_t,X_t(x))dt$ is added to the right-hand side in 
\eqref{eq:RSHE}. We refer to a recent contribution on the subject, see 
\cite{delarueHammersley2024rshe_erg}, for detailed solvability results in this setting. As one may expect, 
the expansion in
Theorem 
	\ref{prop:ito:b}
	remains true, but with an additional term $\int_0^t \int_{\mathbb R} \partial_\mu \varphi(\mu_s)(X_s(x))B(\mu_s,X_s(x)) dx \, ds$
	in the right-hand side.

\subsection{Connection with the existing literature and prospects}

As explained in 
\cite{delarueHammersley2022rshe}, 
our original motivation in the construction of the process $(X_t)_{t \geq 0}$ was to obtain a Markov process, here denoted $(\mu_t)_{t \geq 0}$, 
with values in ${\mathcal P}_2({\mathbb R})$ with a strong Feller Markov semi-group, mapping in positive time bounded (measurable) functions defined on ${\mathcal P}_2({\mathbb R})$ 
onto Lipschitz continuous functions. 
This is the second main result of 
\cite{delarueHammersley2022rshe}, with the (time-dependent) Lipschitz constant returned by 
the semi-group
being integrable  in small time  when
 the parameter $\lambda$ {penalising} the 
 higher modes of the noise is in
 $(1/2,1)$. 
 This example of a probability-measure valued diffusion process is in the {wake} of earlier works on  
 Fleming-Viot processes (see \cite{DawsonMarch,Dawson,Stannat})
 and on 
 Wasserstein diffusions (see \cite{AndresvRenesse,DoringvRenesse,Konarovsky0,Konarovskyi,vRenesseSturm,Sturm}). 
 It should be also regarded (especially in presence of an additional drift $B(\mu_t,X_t(x))$)
as a McKean-Vlasov model with an infinite dimensional common noise, in connection 
with earlier works addressing finite or infinite dimensional common noises (see 
\cite{dawsonVaillancourt1995,KurtzXiong,KurtzXiong2,vaillancourt1988} for a small selection of references). 
In this regard, the It\^o expansion proved here should be seen as one step 
forward towards a comprehensive theory of second-order equation on the space of probability measures. 
For instance, it would be quite straightforward to prove here that the semi-group generated by 
\eqref{eq:RSHE} induces a viscosity solution to the second-order PDE 
driven by ${\mathscr L}$. A more ambitious program would be to study solutions in a more 
classical sense, with smooth or even less regular coefficients. We leave this for the future.

\subsection{Strategy of proof} 

The proof relies explicitly on the scheme introduced in 
\cite{delarueHammersley2022rshe}, which is recalled in 
Section \ref{se:2} right below. 
 {In essence},
we expand the test function $\varphi$ in Theorem 
	\ref{prop:ito:b}
	along this scheme and then pass to the limit 
	as the time discretisation parameter is sent to $0$.
	As a side result of our proof, 
	we establish a bound 
for the derivative 
of $X_t$ in $L^2(\Omega,H^1({\mathbb S}))$ at any $t>0$, which improves the bound 
in $L^2([0,T] \times \Omega,H^1({\mathbb S}))$
used in 
\cite{delarueHammersley2022rshe}
to address the tightness of the scheme. 
This bound is key to obtain, in Section \ref{se:3}, uniform integrability properties 
of the scheme in $H^1$
that are used subsequently in Section \ref{se:4} to 
pass to the limit
in the expansion of $\varphi$ when the mesh of the scheme is sent to $0$.

Throughout, $\| \cdot \|_p$ is the standard $L^p$ norm on ${\mathbb S}$ and
$D$ (or $\nabla$) stands for the standard derivative on ${\mathbb S}$. 
The notations $\lceil \cdot \rceil$ and $\lfloor \cdot \rfloor$ refer respectively to 
the ceil and floor operations. 

\section{{Approximation Scheme}}
\label{se:2}

\subsection{Rearrangement}

Throughout the paper, the rearrangement of an element $f \in L^2({\mathbb S})$ is the unique element  $f^*\in U^2({\mathbb S})$ such that $\textrm{\rm Leb}_{\mathbb S} \circ f = 
\textrm{\rm Leb}_{\mathbb S} \circ f^*$, and $f^* $satisfies the same semi-continuity properties
as those explained in Subsection
\ref{subse:1.1}.
A key inequality is 
(see \cite{Baernstein_correction,baernstein1989convRearrOnCirc} and
\cite[Theorem 8.1]{baernstein2019symmetrizationInAnalysis}):

\begin{lemma}[Riesz rearrangement inequality]
\label{lem:Riesz}
Let $f$, $g$ and $\ell$ be three measurable real-valued functions on 
${\mathbb S}$, such that $\| f \|_p < \infty$, 
$\| g \|_q < \infty$ and 
$\| \ell \|_r < \infty$ for $p,q,r \in [1,\infty]$ with $1/p+1/q+1/r=1$. Then,
\begin{equation*}
\int_{\bS} \int_{\bS} f(x) g(x-y) \ell(y) dx dy \leq \int_{\bS} \int_{\bS} f^*(x) g^*(x-y) \ell^*(y) dx dy.
\end{equation*}
\end{lemma}

We now provide a useful estimate that addresses the concomitant effect of the periodic heat semigroup and the rearrangement operation. 

\begin{lemma}
	\label{lem keyineq}
	Let $u$ belong to $L^2(\bS)$ and $U$ be uniformly distributed on $[0,1]$ on the same space $(\Omega,{\mathcal F},\bP)$ as before. Then, 
\begin{equation}
	\label{eq keyineq}
	\begin{split}
		\E \bigl[\lVert \nabla e^{hU\Delta}u^*\rVert_2^2 \bigr]\leq \E\bigl[\lVert D e^{hU\Delta}u\rVert_2^2\bigr].
	\end{split}
\end{equation}
\end{lemma}

When $h=0$, the result coincides with the so-called 
Pólya–Szeg\"o inequality, see \cite[Theorems 3.6 \& 7.4]{baernstein2019symmetrizationInAnalysis}.

\begin{proof}
Writing $\Gamma_t(x-y)$ 
for the transition kernel, at time $t$ and between $x$ and $y$ in ${\mathbb S}$, 
of the heat kernel on the torus (see (2.1) in \cite{delarueHammersley2022rshe}), we have the following two identities:
\begin{equation*}
	\begin{split}
		\lVert D e^{s \Delta}u^*\rVert_2^2 &=   \int_{\bS}\left(\int_{\bS} D \Gamma_s(x-y) u^*(y)dy\right)^2dx
		\\ 
		&=   \int_{\bS}\int_{\bS} \int_{\bS} D \Gamma_s(x-y)u^*(y)D \Gamma_s(x-z)u^*(z)dydzdx.\\
		\end{split}
		\end{equation*}
		Observe that 
		\begin{equation*}
		\begin{split}
		 \int_{\bS} D \Gamma_s(x-y) D \Gamma_s(x-z) dx = 
		 D_y D_z \int_{\bS}  \Gamma_s(x-y) \Gamma_s(x-z) dx
&		 = D_y D_z \Bigl[ \Gamma_{2s}(y-z) \Bigr]
\\
&= - D^2_x \Gamma_{2s} (y-z).
\end{split}
		\end{equation*}
		Therefore, 
\begin{equation}
\label{eq:D_tGamma}
		\begin{split}
		\lVert D e^{s \Delta}u^*\rVert_2^2
		&=  - \int_{\bS}\int_{\bS}u^*(y)D_x^2 \Gamma_{2s}(y-z)u^*(z)dydz 
		\\
		&=  - \frac12 \int_{\bS}\int_{\bS}u^*(y)D_s \bigl[ \Gamma_{2s}(y-z) \bigr]u^*(z)dydz.\\
	\end{split}
\end{equation}
By integrating the above equality in $s$ over $[0,h]$ and by applying Lemma \ref{lem:Riesz}
(recalling from \cite[Lemma 2.8]{delarueHammersley2022rshe} that $\Gamma_{2s}^*=\Gamma_{2s}$), we get
\begin{equation*}
	\begin{split}
		\int_0^h \lVert D e^{s \Delta}u^*\rVert_2^2 ds 
		&= \frac12  \biggl[ \int_{\bS} u^*(y)^2 dy -  \int_{\bS}\int_{\bS}u^*(y) \Gamma_{2h}(y-z)u^*(z)dydz \biggr] \\
		&\leq  \frac12  \biggl[  \int_{\bS} u (y)^2 dy - \int_{\bS}\int_{\bS} u (y) \Gamma_{2h}(y-z)u (z)dydz \biggr] =  \int_0^h\lVert D e^{s \Delta}u \rVert_2^2 ds,
	\end{split}
\end{equation*}
the last line being obtained by reverting back the computations in \eqref{eq:D_tGamma}.
\end{proof}

\subsection{Scheme} 
\label{subse:3.2}

The construction of a solution to 
\eqref{eq:RSHE}  achieved in \cite{delarueHammersley2022rshe} goes through the scheme:
\begin{equation}
\label{eq scheme n}
	\begin{split}
		X^{h}_{n+1}=  \left(  e^{h\Delta}X^h_{n}+\int_{nh}^{(n+1)h} e^{([n+1]h-s)\Delta}dW_{s} \right)^*, \quad n \in {\mathbb N}_0; 
		\quad X^h_0=  X_0,
	\end{split}
\end{equation}
for a given initial condition $X_0 \in {U^2({\mathbb S})}$
and for a step $h>0$.  
Below, we introduce the following interpolation between times
$nh$ and $(n+1)h$,  
\begin{equation}
\label{eq:Yth}
    Y^h_{t}=e^{(t-nh) \Delta}X^h_{n}+\int^t_{nh} e^{(t-s)\Delta}dW_{s}, \quad 
    t \in [nh,(n+1)h).
\end{equation}
We start with the following lemma:

\begin{lemma} 
\label{lem unif est derivative p1}
Assume that $X_0$ takes values in $H^1({\mathbb S})$ and satisfies 
${\mathbb E} [ \| D X_0 \|_2^{2p}  ] < \infty$
for some $p \geq 1$.
Then, for any real $T>0$, there exists a constant $C_{T,p}$ {(independent of $X_0$)} such that, for any integer $n \geq 0$
such that $n h \leq T$ 
\begin{equation}
	\begin{split}
	\E  & \left[ \left\lVert DX^h_n  \right\rVert_2^{2p }   \right] \leq   
	C_{T,p} \Bigl( 1 + 
	\E   \left[ \left\lVert DX_0  \right\rVert_2^{2p }   \right] \Bigr). 
	\end{split}
\end{equation}
\end{lemma}

\begin{proof}
Following  Lemma \ref{lem keyineq}, 
we introduce a uniformly distributed [0,1]-valued random variable  $\tilde{U}$ constructed on a separate probability space 
$(\tilde \Omega,\tilde{\mathcal F},\tilde{\mathbb P})$ 
and then elevate  
$(  \Omega, {\mathcal F} ,{\mathbb P})$  to a product with $(\tilde \Omega,\tilde{\mathcal F},\tilde{\mathbb P})$. One may denote by 
$\tilde{\mathbb E}$, the integral with respect to $\tilde{\mathbb P}$. 
\vskip 4pt

\noindent \textit{First Step.} 
We 
claim that, for any $m,n,N \in {\mathbb N}_0$, with $m  \geq  N$,  
\begin{equation}
	\label{eq:IV:44}
	\begin{split}
	 \Bigl\|  
		D \Bigl( e^{n h \tilde U \Delta}   X_{m}^h \Bigr) \Bigr\|_2^{2} \leq    \tilde{\mathbb E} 
		\biggl[ \Bigl\|  
		D\Bigl(  e^{ h (m+n -N) \tilde{U} \Delta} X_{N}^h \Bigr) \Bigr\|_2^{2}
		\biggr] + \tilde{\mathbb E} \biggl[ \sum_{k=1}^{m-N} \Bigl( R^{1,m,n}_k + 2 R^{2,m,n}_k \Bigr) \biggr],
	\end{split}
\end{equation}
\vspace{-8pt}
\\
with the two notations {(below, $w_t(x)$ is a shortened notation for the derivative $D   W_t(x)$)}
\begin{equation*} 
\begin{split}
&R^{1,m,n}_k :=  
		\bigg\| 
		e^{ h (n+k-1)  \tilde{U} \Delta} \int_{(m-k)h}^{(m-k+1)h} e^{[(m-k+1)h-s] \Delta}
		d {w}_s
		\biggr\|_2^2
\\
&R^{2,m,n}_k :=		2  \biggl\langle 
		D
		\Bigl( 
		e^{h(1+ (n+ k-1) \tilde U) \Delta} 
		X_{m-k}^h
		\Bigr), 
		\int_{(m-k)h}^{(m-k+1)h} e^{[(m-k+1)h-s] \Delta}
		d {w}_s
		\biggr\rangle. 
	\end{split}
\end{equation*}
The proof of 
\eqref{eq:IV:44} is by backward induction on $m$, starting from $m=N$. 
In the latter case, the bound is obvious. Now, if the bound 
is true for some 
{$m \geq N$}
 and any 
$n \in {\mathbb N}_0$, then one
may use   Lemma
\ref{lem keyineq} and the contraction property of the heat semi-group to get  
\begin{equation*}
	\begin{split}
		 \Bigl\|  
		D \Bigl( e^{n h \tilde U \Delta}   X_{m+1}^h \Bigr) \Bigr\|_2^{2} 
		&= \bigg\| 
		D
		\biggl[ e^{n h \tilde U \Delta}
		\biggl( 
		e^{h \Delta} 
		X_{m}^h
		+
		\int_{mh}^{(m+1)h} e^{[(m+1)h-s] \Delta}
		d W_s
		\biggr)^* \biggr]
		\biggr\|_2^2
		\\
		&\leq \bigg\| 
		D
		\biggl[ e^{n h \tilde U \Delta}
		\biggl( 
		e^{h \Delta} 
		X_{m}^h
		+
		\int_{m h}^{(m+1)h} e^{ [(m+1)h-s]\Delta}
		d W_s
		\biggr) \biggr]
		\biggr\|_2^2
		\\
		&\leq 
		\tilde{\mathbb E} 
		\Bigl[ \bigl\|  
		D e^{ ( n +1) h \tilde{U} \Delta} X_{m}^h \bigr\|_2^{2}
		\Bigr]
		+
		\bigg\| 
		e^{ n  h  \tilde{U} \Delta}
		\int_{m h}^{(m+1)h} e^{ [(m+1)h-s] \Delta}
		d {w}_s
		\biggr\|_2^2 
		\\
		&\hspace{15pt} +
		2 \biggl\langle 
		D
		\Bigl( 
		e^{h ( 1 + n \tilde U) \Delta} 
		X_{m}^h
		\Bigr), 
		\int_{m h}^{(m+1)h} e^{[(m+1)h-s] \Delta}
		d {w}_s
		\biggr\rangle, 
	\end{split}
\end{equation*}
Identity 
	\eqref{eq:IV:44} follows by induction. 
\vskip 4pt

\noindent \textit{Second Step:}
Estimate the $p^{th}$ moments of $R^{1,m}:=\tilde{\mathbb E}[\sum_{k=1}^{m-N} R^{1,m,0}_k]$ and \\ $R^{2,m}:=\tilde{\mathbb E}[\sum_{k=1}^{m-N} R^{2,m,0}_k]$. For $R^{1,m}$, 
note from 
the equivalence $\lambda_m \sim m^{-\lambda}$, that for $k \geq 2$,  
\begin{equation*}
\begin{split}
&\tilde{\mathbb E} 
\bigl[ R^{1,m,0}_k \bigr]
\leq c \tilde{\mathbb E} \sum_{\ell \in {\mathbb N}} 
\biggl(
e^{-4 \pi^2 h (k-1)  \tilde{U} \ell^2} 
\int_{(m-k)h}^{(m-k+1)h} 
\ell^{(1- \lambda)}
e^{-4 \pi^2[(m-k+1) h-s]  \ell^2} 
dB_s^{\ell}
\biggr)^2, 
\end{split}
\end{equation*}
and, by Jensen's inequality (\textcolor{black}{in the form 
$(\tilde{\mathbb E}\sum_{\ell \in {\mathbb N}} 
x_\ell^2 y_\ell^2)^p 
\leq ( \tilde{\mathbb E} \sum_{\ell \in {\mathbb N}} x_\ell^2 )^{p-1}
\tilde{\mathbb E} ( \sum_{\ell \in {\mathbb N}} x_\ell^2 y_\ell^{2p})$}), 
\begin{equation*}
\begin{split}
\Bigl(
\tilde{\mathbb E} 
\bigl[ R^{1,m,0}_k \bigr]
\Bigr)^p
&\leq c \biggl( 
\tilde{\mathbb E}
\sum_{\ell \in {\mathbb N}} 
\ell^{2(1-\lambda)}
e^{-8 \pi^2 (k-1) h \tilde{U} \ell^2} 
\biggr)^{p-1} 
\\
&\hspace{15pt} \times 
\tilde{\mathbb E} \biggl[ \sum_{\ell \in {\mathbb N}} 
\ell^{2(1-\lambda)}
e^{-8 \pi^2 (k-1) h \tilde{U} \ell^2} 
\biggl(\int_{(m-k)h}^{(m-k+1)h} 
e^{-4 \pi^2 [(m-k+1) h-s]  \ell^2} 
dB_s^{\ell}
\biggr)^{2p} \biggr].
\end{split}
\end{equation*}
Therefore, 
\begin{equation*}
\begin{split}
{\mathbb E} 
\Bigl[ 
\Bigl(
\tilde{\mathbb E} 
\bigl[ R^{1,m,0}_k \bigr]
\Bigr)^p
\Bigr] 
&\leq c \biggl( 
\tilde{\mathbb E}
\sum_{\ell \in {\mathbb N}} 
\ell^{2(1-\lambda)}
e^{-8 \pi^2 (k-1) h \tilde{U} \ell^2} 
\biggr)^{p-1} 
\\
&\hspace{-15pt} \times 
\tilde{\mathbb E} \biggl[ \sum_{\ell \in {\mathbb N}} 
\ell^{2(1-\lambda)}
e^{-8 \pi^2 (k-1) h \tilde{U} \ell^2} 
{\mathbb E} 
\biggl\{
\biggl(\int_{(m-k)h}^{(m-k+1)h} 
e^{-4 \pi^2 [(m-k+1) h-s]  \ell^2} 
dB_s^{\ell}
\biggr)^{2p} 
\biggr\} \biggr]
\\
&\leq c h^p
\biggl( 
\tilde{\mathbb E}
\sum_{\ell \in {\mathbb N}} 
\ell^{2(1-\lambda)}
e^{-8 \pi^2 (k-1) h \tilde{U} \ell^2} 
\biggr)^{p}.
\end{split}
\end{equation*}
Using the 
Gaussian estimate
$ \sum_{m\geq 1}   e^{-8 \pi^2 khUm^2}  m^{2(1-\lambda)_+}     \leq
 \int_0^{\infty}  e^{- 8 \pi^2  k h U x^2}   ( x^{2(1-\lambda)_+} +1 ) dx
 \leq c_\lambda [ (khU)^{-\frac12} + (khU)^{-\frac12 - (1-\lambda)_+} ]$, 
we get, for $2 \leq k \leq m-N$,
\begin{equation*}
\begin{split}
\biggl( \tilde{\mathbb E} \sum_{\ell \in {\mathbb N}} 
\ell^{2(1-\lambda)}
e^{-8 \pi^2 (k-1) h \tilde{U} \ell^2} 
\biggr)^p
&\leq c_{p,\lambda}  \Bigl\{ 1 + \bigl( (k-1) h  \bigr)^{-\tfrac{1}2-(1-\lambda)_+} \Bigr\}^p,
\end{split}
\end{equation*}
from which we deduce that, for $2 \leq k \leq m-N$, 
\begin{equation}
\label{eq:R1k}
\begin{split}
&
{\mathbb E} 
\Bigl[ 
\Bigl(
\tilde{\mathbb E} 
\bigl[ R^{1,m,0}_k \bigr]
\Bigr)^p 
\Bigr] 
\leq c_{p,\lambda} 
h^p 
\Bigl\{ 1+ 
\bigl( (k-1) h  \bigr)^{-\tfrac{1}2-(1-\lambda)_+} \Bigr\}^p.
\end{split}
\end{equation}
 When $k=1$ (and $m-N \geq 1$), 
 \begin{equation*}
\begin{split}
&{\mathbb E} 
\Bigl[ 
\Bigl( \tilde{\mathbb E} 
\bigl[ R^{1,m,0}_1 \bigr] \Bigr)^p
\Bigr] 
\leq 
c {\mathbb E} 
\biggl[ 
\biggl\{
\sum_{\ell \in {\mathbb N}} 
\ell^{-   (\lambda+\frac12)}
\biggl(
\int_{(m-1)h}^{m h} 
\ell^{\frac{5}4- \frac{\lambda}2}
e^{- 4 \pi^2 (m h-s)  \ell^2} 
dB_s^{\ell}
\biggr)^2 \biggr\}^p  \biggr].
\end{split}
\end{equation*}
By Jensen's inequality, 
 \begin{align}
{\mathbb E} 
\Bigl[ 
\Bigl( \tilde{\mathbb E} 
\bigl[ R^{1,m,0}_1 \bigr] \Bigr)^p
\Bigr] 
&\leq   c_{p,\lambda}
\sum_{\ell \in {\mathbb N}}
\ell^{- (\lambda+\frac12)}
\ell^{(\frac5{2}- \lambda) p} {\mathbb E} 
\biggl[ 
\biggl(
\int_{(m-1)h}^{m h} 
e^{- 4 \pi^2 (m h-s)  \ell^2} 
dB_s^{\ell}
\biggr)^{2p}  \biggr] \nonumber
\\
&\leq 
 c_{p,\lambda}
\sum_{\ell \in {\mathbb N}}
\ell^{- (\lambda+\frac12)}
\ell^{(\frac5{2}- \lambda) p} 
\biggl( 
\int_{0}^{ h} 
e^{- 8 \pi^2 s  \ell^2} 
ds
\biggr)^{p}
 \label{eq:R11}
\\
&\leq  c_{p,\lambda} \sum_{\ell \in {\mathbb N}} \ell^{- (\lambda+\frac12)} 
\ell^{ (\frac12 - \lambda)p}
\bigl( 1 - \exp(- 8 \pi^2  h \ell ^2) \bigr)^p
 \leq  
 c_{p,\lambda} 
h^{\min(\lambda-\frac12,1) \textcolor{black}{\frac{p}2}}, \nonumber
\end{align}
where we used the inequality $1-\exp(-x) \leq c_\lambda x^{\min(\lambda-\frac12,1)\textcolor{black}{\frac12}}$.  Therefore, by defining $b_{k,h}: = 1 + ((k-1) h)^{-\tfrac{1}2-(1-\lambda)_+}$,
we deduce that, from \eqref{eq:R11} and \eqref{eq:R1k},
\begin{equation*}
	\begin{split}
%
	  {\mathbb E} \bigl[ (R^{1,m})^p   \bigr]  
		&\leq c_p
		{\mathbb E} 
\Bigl[ 
\Bigl( \tilde{\mathbb E} 
\bigl[ R^{1,m,0}_1 \bigr] \Bigr)^p
\Bigr] 
+
c_p	{\mathbb E} 
\biggl[ 
\biggl( \tilde{\mathbb E} 
\biggl[ \sum_{k=2}^{m-N} R^{1,m,0}_k \biggr] \biggr)^p
\biggr]
\\
&\leq			
c_{p,\lambda} h^{\min(\lambda-\frac12,1) \textcolor{black}{\frac{p}2}}
+ 
c_p	 
\biggl(
\sum_{k=2}^{m-N} b_{k,h} \biggr)^{p-1} 
\biggl( 
\sum_{k=2}^{m-N} b_{k,h} b_{k,h}^{-p}
{\mathbb E} 
\Bigl[
\Bigl( \tilde{\mathbb E} \bigl[ R^{1,m,0}_k \bigr] \Bigr)^p
\Bigr] \biggr)
		\\
		&\leq 
c_{p,\lambda} h^{\min(\lambda-\frac12,1) \textcolor{black}{\frac{p}2}}
+ 
c_{p,\lambda}
h^p
\biggl(
\sum_{k=2}^{m-N} b_{k,h} \biggr)^{p}.
	\end{split}
\end{equation*}
Noting that $ 
\sum_{k\geq 2}^{m-N} 
[b_{k,h}-1]\color{black}= \sum_{k\geq 1}^{m-N-1} 
( k h)^{-\tfrac{1}2-(1-\lambda)_+}
  \leq 
  h^{-1} \int_0^{h (m-N)}
x^{-\frac{1}2-(1-\lambda)_+}
dx$, we get 
\begin{equation}
\label{eq:sum:bkh}
	\begin{split}
%
			 & {\mathbb E} \bigl[ (R^{1,m})^p   \bigr]  
 \leq 
c_{p,\lambda} h^{\min(\lambda-\frac12,1) \textcolor{black}{\frac{p}2}}
+ 
c_{p,\lambda}
h^p
\biggl(
m-N + h^{-1} \int_0^{h (m-N)}
x^{-\frac{1}2-(1-\lambda)_+}
dx
 \biggr)^{p}
 \\
 &\leq 
 c_{p,\lambda} h^{\min(\lambda-\frac12,1) \textcolor{black}{\frac{p}2}}
+ 
c_{p,\lambda,T}
\bigl( h (m-N) 
 \bigr)^{p (\frac12-(1-\lambda)_+)}
 \leq 
c_{p,\lambda,T}
\bigl( h (m-N) 
 \bigr)^{ \min(\lambda-\frac12,\frac12)\textcolor{black}{\frac{p}2}}, 
 \end{split}
 \end{equation} 
with the last line following from 
$\frac12 -(1-\lambda)_+
= \min(\lambda- \frac12 ,\frac12)$ together with the bound 
$h m \leq T+1$. 
This gives a bound for $R^{1,m}$.
\vskip 4pt

\noindent \textit{Third Step.}
For the treatment of $R^{2,m}$, 
using the symmetry of the Laplace operator and 
the
Burkholder-Davis-Gundy inequality (see Theorem 4.36 in 
\cite{daPratoZabczyk2014stochEqnsInfDim}), one may write
\begin{align}
&{\mathbb E} 
\Bigl[ 
\bigl( 
 R^{2,m} \bigr)^p
\Bigr] \nonumber
		\\
		&=
		{\mathbb E} 
		\Biggl[ \Biggl(
		\sum_{k=1}^{m-N}
		\biggl[ \Bigl\langle 
		 D
		X_{m-k}^h
, 
\tilde{\mathbb E}		\int_{(m-k)h}^{(m-k+1)h} e^{[(m-k+1)h
		+
		h(1+ (k-1) \tilde U)  
		-s] \Delta}
		d   w_s
		\Bigr\rangle \biggr]
		\Biggr)^p   \Biggr] \nonumber
\\
		&\leq 
		c_p
		{\mathbb E} 
		\Biggl[ \Biggl(
		\sum_{k=1}^{m-N}
		\biggl[ 
		\bigl\|
		 D
		X_{m-k}^h
\bigr\|_2^2
 \Bigl[ \int_{(m-k)h}^{\cdot}
\tilde{\mathbb E}	\Bigl(	 e^{((m-k+1)h
		+
		h(1+ (k-1) \tilde U)  
		-s) \Delta}
		\Bigr)
		d  w_s
		\Bigr]_{(m-k+1)h}
	 \biggr]
		\Biggr)^{p/2}  \Biggr] \nonumber
		\\
		&\leq c_{p}
		{\mathbb E} 
		\Biggl[
		\Biggl(
		 \sum_{k=1}^{m-N}
		\Bigl\{
		 \bigl\|
		D X_{m-k}^h
		 \bigr\|_2^2 
		 {\mathbb E}
		 \Bigl(
		\tilde{\mathbb E} 
\bigl[ R^{1,m,0}_k \bigr]
\Bigr)
\Bigr\}
 	\Biggr)^{p/2}  \Biggr],
		\label{eq u2 gen mean}
\end{align}
\textcolor{black}{where we used stochastic Fubini's theorem to pass the expectation symbol 
$\tilde{\mathbb E}$ in the fourth line}.
{In the third line, $[ M_\cdot ]_t$ denotes the bracket at time $t$ of a martingale process $M_\cdot$.}
We split
\eqref{eq u2 gen mean}
as follows: %
\begin{equation*}
\begin{split}
&{\mathbb E} 
\Bigl[ 
\bigl( R^{2,m} \bigr)^p
\Bigr]
\\
& \leq  c_{p}
		{\mathbb E} 
		\Bigl[
		 \bigl\|
		D X_{m-1}^h
		 \bigr\|_2^p 
		 \Bigr]
		 {\mathbb E}
		 \Bigl(
		\tilde{\mathbb E} 
\bigl[ R^{1,m,0}_1 \bigr]
\Bigr)^{p/2}
+
c_p		{\mathbb E} 
		\Biggl[
		\Biggl(
		 \sum_{k=2}^{m-N}
		\Bigl\{
		 \bigl\|
		D X_{m-k}^h
		 \bigr\|_2^2 
		 {\mathbb E}
		 \Bigl(
		\tilde{\mathbb E} 
\bigl[ R^{1,m,0}_k \bigr]
\Bigr)
\Bigr\}
 	\Biggr)^{p/2}  \Biggr].	
\end{split}
\end{equation*}
Inserting
  \eqref{eq:R1k} and 
   \eqref{eq:R11}
   and recalling from
   \eqref{eq:sum:bkh}
   that 
   $h \sum_{k=2}^{m-N} b_{k,h} \leq
   c_{\lambda,T}$
  \begin{equation*}
\begin{split}
{\mathbb E} 
\Bigl[ 
\bigl( R^{2,m} \bigr)^p
\Bigr]
&\leq  c_{p,\lambda}
		{\mathbb E} 
		\Bigl[
		 \bigl\|
		D X_{m-1}^h
		 \bigr\|_2^p 
		 \Bigr]
h^{\min(\lambda-\frac12,1) \textcolor{black}{\frac{p}4}}
\\
&\hspace{15pt} +
c_{p,\lambda}		{\mathbb E} 
		\Biggl[
		\Biggl(h 
		 \sum_{k=2}^{m-N}
		\Bigl\{
		 \bigl\|
		D X_{m-k}^h
		 \bigr\|_2^2 
\Bigl( 1+ 
\bigl( (k-1) h  \bigr)^{-\tfrac{1}2-(1-\lambda)_+} \Bigr)
\Bigr\}
 	\Biggr)^{p/2}  \Biggr]
	\\
&\leq  c_{p,\lambda}
		{\mathbb E} 
		\Bigl[
		 \bigl\|
		D X_{m-1}^h
		 \bigr\|_2^p 
		 \Bigr]
h^{\min(\lambda-\frac12,1) \textcolor{black}{\frac{p}4}}
\\
&\hspace{15pt} +
c_{p,\lambda}		h 
		 \sum_{k=2}^{m-N}
\Bigl( 1+ 
\bigl( (k-1) h  \bigr)^{-\tfrac{1}2-(1-\lambda)_+} \Bigr)
		{\mathbb E} 
		\Bigl[
		 \bigl\|
		D X_{m-k}^h
		 \bigr\|_2^{p} 
\Bigr].	
\end{split}
\end{equation*}
Collect the above bounds for $R^{2,m}$ 
and the bound \eqref{eq:sum:bkh} for $R^{1,m}$ and plug them into \eqref{eq:IV:44}. 
Assuming that $(m-N)h \leq \varepsilon$ for some 
$\varepsilon \in [h,1)$ (which \textcolor{black}{is possible since} $h<1$),
recalling that $\frac12 - (1-\lambda)_+
= \min(\frac12,\lambda-\frac12)$ and repeating 
\eqref{eq:sum:bkh}, 
one has
\begin{equation*}
\begin{split}
&\sup_{m \geq N : h (m-N) \leq \varepsilon}
	{\mathbb E} 
	\Bigl[ 
	\bigl\|  
	D X_{m}^h\bigr\|_2^{2p}
	\Bigr] 
	\\
&\hspace{15pt}	\leq 
	c_{p,\lambda,T} \biggl( 1+ 
		{\mathbb E} 
	\Bigl[ 
	\bigl\|  
	D X_{N}^h\bigr\|_2^{2p} \Bigr]
	+
	\varepsilon^{\min(\frac12,\lambda-\frac12)\textcolor{black}{\frac12}}
	\sup_{m \geq N : h (m-N) \leq \varepsilon}
	{\mathbb E} 
	\Bigl[ 
	\bigl\|  
	D X_{m}^h\bigr\|_2^{2p}
	\Bigr]
	\biggr),
	\end{split}
\end{equation*}
from which we deduce that, for $\varepsilon \leq \varepsilon_{p,\lambda,T}$ small enough (the threshold
$\varepsilon_{p,\lambda,T}$ being strictly positive and only depending on 
$p$, $\lambda$ and $T$),
\begin{equation*}
\begin{split}
&\sup_{m \geq N : h (m-N) \leq \varepsilon}
	{\mathbb E} 
	\Bigl[ 
	\bigl\|  
	D X_{m}^h\bigr\|_2^{2p}
	\Bigr] 
	\leq 
	c_{p,\lambda,T} \Bigl( 1+ 
		{\mathbb E} 
	\Bigl[ 
	\bigl\|  
	D X_{N}^h\bigr\|_2^{2p}
	\Bigr] \Bigr),
	\end{split}
\end{equation*}
observing by induction that 
${\mathbb E}  [  \|	D X_{m}^h \|_2^{2p}
	]$
	is finite for $m >N$.
	(Notice that, when $h\geq  \varepsilon_{p,\lambda,T}$, 
	the above is obviously true since the set of indices in the left-hand side 
	reduces to the singleton $\{N\}$.)
	{The result follows by repeating the argument on intervals of length $\varepsilon_{p,\lambda,T}$}.
\end{proof}

\section{Uniform integrability of the scheme in $H^1$}
\label{se:3}

We let
$(\tilde X^h_t := (\lceil \nicefrac{t}{h} \rceil - \nicefrac{t}{h}  )X^h_{\lfloor \nicefrac{t}{h} \rfloor} + (\nicefrac{t}{h}-\lfloor \nicefrac{t}{h} \rfloor )X^h_{\lceil \nicefrac{t}{h} \rceil}
)_{t \geq 0}$
be the interpolation of   $(X^h_n)_{n \in {\mathbb N}_0}$. 
By 
\cite[Theorem 4.15]{delarueHammersley2022rshe}, 
$(\tilde X^h_t,W_t)_{t \geq 0}$
converges
in law
over 
$\cC([0,\infty), {H_{\rm sym}^{-1}({\mathbb S} )
\times 
L^2_{\rm sym}({\mathbb S} )})$
(space of continuous functions from $[0,\infty)$ to 
$H^{-1}_{\rm sym}({\mathbb S}) \times L^2_{\rm sym}({\mathbb S})$,
  equipped with the uniform topology on compact subsets) to 
  the  pair 
    $(X_t,W_t)_{t \geq 0}$ in \eqref{eq:RSHE} 
with $X_0$ as initial condition. 
For any $p \geq 1$ and $T >0$,  
\begin{equation}
\label{eq:Lp:tilde:Xh}
\sup_{h \in (0,1)} 
{\mathbb E} \bigl[\sup_{0 \le t \le T} \|\tilde X^h_t \|_2^{2p}
\bigr]
< \infty.
\end{equation} 
 
We draw two  consequences from 
\eqref{eq:Lp:tilde:Xh}.
Back to 
\eqref{eq:Yth}, we first notice that, for any $T>0$, 
\begin{equation*}
\begin{split} 
\sup_{s \in [0,T]}
\| Y_s^h - \tilde X_s^h \|_{2,-1} 
&\leq
\sup_{k \leq \lfloor  \nicefrac{T}{h} \rfloor}
\sup_{t \in [kh,(k+1)h]}
\Bigl\|
e^{(t-nh) \Delta}X^h_{n}+\int^t_{nh} e^{(t-s)\Delta}dW_{s}
- \tilde X_s^h
\Bigl\|_{2,-1}.
\end{split} 
\end{equation*}
Fix $A>0$. On the event 
$E^A
:=
\{ \sup_{0 \le t \le T} \| \tilde X^h_t \|_2 \leq A \} 
\cap 
\{ \sup_{0 \le t \le T} \| \int^t_{0} e^{(t-s)\Delta}dW_{s}  \|_2 \leq A \}$,
\begin{equation*}
\begin{split} 
\sup_{s \in [0,T]}
\| Y_s^h - \tilde X_s^h \|_{2,-1} 
& \leq 
2 \sup_{\| f \|_2 \leq A }
\sup_{s \in [0,h]}
\| e^{s \Delta} f - f\|_{2,-1} 
+ 
\sup_{\vert t-s \vert \leq h} 
\| \tilde X_t^h - 
\tilde X_s^h
 \|_{2,-1} 
\\
&\hspace{15pt} + 
\sup_{\vert t-s \vert \leq h} 
\Bigl\| \int^t_{0} e^{(t-r)\Delta}dW_{r}  - 
\int^s_{0} e^{(s-r)\Delta}dW_{r} 
 \Bigr\|_{2,-1}.
\end{split} 
\end{equation*}
By \eqref{eq:Lp:tilde:Xh}
and the fact  that the processes 
$(\tilde X^h_s)_{0 \leq s \leq T}$ are tight on 
${\mathcal C}([0,T];H^{-1}({\mathbb S}))$, we get
(this is our first direct consequence of \eqref{eq:Lp:tilde:Xh}):
\begin{equation}
\label{eq:convergence:P:probability}
\lim_{h \rightarrow 0} 
\sup_{s \in [0,T]}
\| Y_s^h - \tilde X_s^h \|_{2,-1} 
\underset{{\mathbb P}{\rm -probability}}= 0.
\end{equation} 
Here is now the second consequence. From \eqref{eq:Yth}, there exists a (universal) constant $c$ such that, 
for any $n \in {\mathbb N}_0$ and any 
$t \in [nh,(n+1) h)$, 
\begin{equation*}
\label{eq:proof:lem:25:04:5}
\begin{split} 
&{\mathbb E} 
\Bigl[
 \|  Y^h_{t} \|_2^4 + 
 \|D Y^h_t \|_2^4 \Bigr]  
 \leq 
 c 
 {\mathbb E} 
 \biggl[ \|  X^h_{n}  \|^4_2 
+ \| D  X^h_n  \|_2^4 
+
\Bigl\| 
\int_{nh}^{t} 
 e^{(t-s)\Delta}dW_{s}
 \Bigr\|_2^4 
 +
\Bigl\| 
\int_{nh}^{t} 
 e^{(t-s)\Delta}dw_{s}
 \Bigr\|_2^4  
 \biggr],
 \end{split} 
\end{equation*}
where 
$(w_t=DW_t)_{t \geq 0}$. 
Using 
\eqref{eq:Lp:tilde:Xh}
and 
Lemma \ref{lem unif est derivative p1},
and
following  
 \eqref{eq:R11}, 
 we get
 \begin{equation} 
\label{eq:L4:bounds} 
 \sup_{h \in (0,1)} \sup_{t \in [0,T]}
 {\mathbb E} 
\Bigl[
 \|  Y^h_{t} \|_2^4 + 
 \|D Y^h_t \|_2^4 \Bigr]  
 < \infty. 
 \end{equation}
 
Now, as a consequence of  Lemma \ref{lem unif est derivative p1}, we obtain 
(recalling again the notation 
\eqref{eq:Yth})
\begin{lemma}
\label{lem:25:04:6}
We assume that $X_0$ takes values in $H^1({\mathbb S})$ and satisfies 
${\mathbb E}[ \| D X_0 \|_2^{2p}] < \infty$ for any $p \geq 1$.
Denoting by
$(\widehat{Y}_t^{h,k})_{k \in {\mathbb N}_0}$
the Fourier modes of $Y^h_t$,  i.e. $Y^h_t 
:= \sum_{k \in {\mathbb N}_0}
	\widehat{Y}_t^{h,k} e_k
$, 
let
$Y_t^{h,N} := \sum_{0 \le k \le N}
	\widehat{Y}_t^{h,k} e_k$ for any $N \in {\mathbb N}_0$.
Then, for any $T,\epsilon >0$, there
exists  $N_\epsilon$
such that 
\begin{equation}
\label{eq:Y:YN:distance}
	\limsup_{h \rightarrow 0}
	\biggl[ {\mathbb E} \int_0^T
	\| 
	D Y^h_s - DY^{h,N_\epsilon}_s\|^2_2
	ds \biggr]
	+
	\sup_{h \in (0,1)}
	\sup_{t \in [0,T]}
	{\mathbb E} 
	\Bigl[ 	\sup_{x \in {\mathbb S}} 
\vert Y^h_t(x) - Y^{h,N_\epsilon}_t(x) |^2
	\Bigr]
	\leq \varepsilon.
\end{equation} 
\end{lemma}
\begin{proof}
%
By Sobolev embedding theorem, there exists a constant $c>0$ such that, for any $h>0$, 
$t \in [0,T]$,  $N \in {\mathbb N}$ and $A>0$,  
\begin{equation*} 
\begin{split} 
&	{\mathbb E} 
	\Bigl[ 
	\sup_{x \in {\mathbb S}} 
\vert Y^h_t(x) - Y^{h,N}_t(x) |^2
\Bigr] 
 \leq c \sup_{u \in H^1_{\rm sym}({\mathbb S}) : \| u \|_{2,1} \leq A}
\| u - u ^N \|_{2,1}^2
+  
{\mathbb E} 
	\Bigl[ 
	 \| Y^h_t \|_{2,1}^2
	{\mathbf 1}_{\{  \| Y^h_t \|_{2,1} \geq A\}}
\Bigr].
\end{split}
\end{equation*} 
For a given $\varepsilon >0$, we can choose $A$ sufficiently large such that the second term in the right-hand side is less than $\varepsilon/2$, uniformly with respect 
to $h>0$ and $t \in [0,T]$. For this value of $A>0$, we may choose $N$ large enough that the first term in the left-hand side is also less than $\varepsilon/2$. 
This proves the claim for the second term in the left-hand side of 
\eqref{eq:Y:YN:distance}. 

In order to show the claim for the first term, it suffices to prove that 
\begin{equation}
\label{eq:DYsh:DXs}
\limsup_{h \rightarrow 0} {\mathbb E} \int_0^T
\bigl\| D Y_s^h \bigr\|_2^2 ds
\leq {\mathbb E} \int_0^T
\bigl\| D  X_s \bigr\|_2^2 ds,
\end{equation}
for $X$ the solution to 
\eqref{eq:RSHE} 
with $X_0$ as initial condition. 
Indeed, if true, 
\eqref{eq:DYsh:DXs}
would  imply 
\begin{equation*}
\begin{split}
\limsup_{h \rightarrow 0} {\mathbb E} \int_0^T
\bigl\| D Y_s^h - D Y_s^{h,N} \bigr\|_2^2 ds
&= 
\limsup_{h \rightarrow 0} {\mathbb E} \int_0^T
\bigl( \bigl\| D Y_s^h \bigr\|_2^2 - \bigl\| D Y_s^{h,N} \bigr\|_2^2 \bigr) ds
\\
&\leq
{\mathbb E} \int_0^T
\bigl( \bigl\| D  X_s \bigr\|_2^2
- 
\bigl\| D  X_s^N \bigr\|_2^2 \bigr) ds,
\end{split}
\end{equation*}
and  
\eqref{eq:Y:YN:distance}
would follow by observing that the right-hand tends to $0$ as $N$ tends to $\infty$. 
Notice that, above, the fact that 
$\lim_{h \rightarrow 0} {\mathbb E} \int_0^T
  \| D Y_s^{h,N}\|_2^2 ds=
  {\mathbb E} \int_0^T
  \| D X_s^N \|_2^2 ds$ for any fixed $N$ is a consequence of 
  \eqref{eq:convergence:P:probability} (which says in particular that each Fourier mode of 
$Y_s^h$ converges in law to the corresponding Fourier mode of $X_s$) 
and of the uniform integrability property
\eqref{eq:L4:bounds}.

The proof of 
\eqref{eq:DYsh:DXs}
follows from the proof of 
Corollary 4.12  in \cite{delarueHammersley2022rshe}, see in particular the derivation of (4.27) therein. Namely, there exists a constant $c_{1,\lambda}$ such that 
(for $nh \le T$)
\begin{equation*} 
{\mathbb E} 
\bigl[ \| X_{n+1}^h \|_2^2 \bigr] 
+ 2 {\mathbb E} \biggl[ \int_{0}^{h} 
\| D e^{s \Delta} X_n^h \|_2^2 
ds\biggr] \leq  
{\mathbb E} 
\bigl[  \| X_{n}^h \|_2^2 \bigr] 
+ c_{1,\lambda} h + h \delta(h), 
\end{equation*}
with $\lim_{h \rightarrow 0} \delta(h)=0$. 
Next, we insert the definition of 
$Y_t^h$ and notice from a standard independence argument that 
the bound remains true if one replaces 
$\| D e^{s \Delta} X_n^h \|_2$ by 
$ \| D Y_s^h \|_2$ in the left-hand side. 
Fixing $T>0$ as in the statement, letting $N_T^h:=\lfloor T/h \rfloor+h \geq T$
and 
summing over $n$ between $0$ and $N_T-1$, we deduce that 
\begin{equation*}
 2 {\mathbb E} \biggl[ \int_{0}^{T} 
\| D Y_s^h \|_2^2 
ds\biggr] \leq  
{\mathbb E} 
\bigl[  \| X_0 \|_2^2 \bigr] 
-
{\mathbb E} 
\bigl[ \| X_{N_T^h+h}^h \|_2^2 \bigr] 
+ c_{1,\lambda} (T+\delta(h)). 
\end{equation*} 
By
\eqref{eq:Y:YN:distance}
(and the weak convergence of the Fourier modes), 
$\lim_{h \rightarrow 0} {\mathbb E} [ \| X_{N_T^h+h}^h \|_2^2 ] 
=
{\mathbb E} [ \| X_T\|^2_2]$.
Inequality 
\eqref{eq:DYsh:DXs}
follows, recalling 
from 
\cite[Corollary 4.12 \& (4.31)]{delarueHammersley2022rshe}
that \\
${\mathbb E} 
[ \| X_0 \|_2^2 ] 
-
{\mathbb E} 
[ \| X_T \|_2^2 ] 
+ c_{1,\lambda} T  
= 2\int_0^T 
{\mathbb E}  [ 
\| D X_s \|^2_2 ] ds.$
\end{proof}

We end up with 
\begin{lemma}
\label{lem:25:04:7}
We assume that $X_0$ takes values in $H^1({\mathbb S})$ and satisfies 
${\mathbb E} [ \| D X_0 \|_2^{2p}] < \infty$ for any $p \geq 1$. 
Then, for any $T>0$ and any (bounded and) continuous function $\psi$ from ${\mathcal P}_2({\mathbb R}) \times {\mathbb R}$ to ${\mathbb R}$, 
the random variables
\begin{equation*}
\biggl( \int_0^T 
\int_{\mathbb S} 
\psi \bigl( \textrm{\rm Leb}_{\mathbb S} \circ (Y_s^h)^{-1} ,Y_s^h(x)\bigr) [D  Y^h_s(x)]^2 dx \, ds
\biggr)_{0 < h < 1}
\end{equation*}
  converge in law, as $h$ tends to $0$, 
  to 
$\displaystyle \int_{\mathbb S} 
\psi \bigl( \textrm{\rm Leb}_{\mathbb S} \circ X_s^{-1},X_s(x)\bigr) [ D X_s(x)]^2 dx \, ds$.
\end{lemma} 

\begin{proof}
By weak convergence of 
$((\tilde X^h_s)_{0 \leq s \leq T})_{h \in (0,1)}$
in
${\mathcal C}([0,T];H^{-1}({\mathbb S}))$
and by 
\eqref{eq:convergence:P:probability}
(which implies the weak convergence of
$((Y_s)_{0 \leq s \leq T})_{h \in (0,1)}$, we can easily pass to the (weak limit) on  
\begin{equation*}
\int_0^T 
\int_{\mathbb S} 
\psi \bigl( \textrm{\rm Leb}_{\mathbb S} \circ (Y_s^{h,N})^{-1},Y_s^{h,N}(x)\bigr) [ D Y_s^{h,N}(x)]^2 dx \, ds
\end{equation*}
as $h$ is sent to $0$, 
for any fixed $N$. Indeed, the above can be written as a continuous function
of   
the Fourier coefficients $((\widehat{Y}_s^{h,k})_{0 \leq k \le N})_{0 \le s \le T}$. 
\color{black}
It remains to see that there exists a constant $C_\psi$, depending on $\psi$, such that 
\begin{align}
\nonumber
&{\mathbb E} 
\biggl\vert 
\int_0^T 
\int_{\mathbb S} 
\psi \bigl( \textrm{\rm Leb}_{\mathbb S} \circ (Y_s^h)^{-1},Y_s^h(x)\bigr) [ D Y_s^{h}(x)]^2 dx \, ds
\\
&\hspace{15pt} -\int_0^T 
\int_{\mathbb S} 
\psi \bigl(  \textrm{\rm Leb}_{\mathbb S} \circ (Y_s^{h,N})^{-1},Y_s^{h,N}(x)\bigr) [ DY_s^{h,N}(x)]^2 dx \, ds
\biggr\vert 
\label{eq:Y:-YN}
\\
&\leq
C_\psi 
{\mathbb E} 
\biggl[
\int_0^T 
\int_{\mathbb S} 
\Bigl\vert [ D Y_s^{h}(x)]^2 - [ D Y_s^{h,N}(x)]^2 \Bigr\vert \, dx \, ds
\biggr] \nonumber
\\
&\hspace{15pt} 
+
{\mathbb E} 
\biggl[
\int_0^T 
\sup_{x \in {\mathbb S}} 
\vert 
\psi \bigl( \textrm{\rm Leb}_{\mathbb S} \circ (Y_s^h)^{-1},Y_s^h(x)\bigr)
-
\psi \bigl(  \textrm{\rm Leb}_{\mathbb S} \circ (Y_s^{h,N})^{-1},Y_s^{h,N}(x)\bigr) 
\vert
\| D Y_s^h \|_2^2. \nonumber 
\, ds \biggr].
\end{align}
By combining  Lemma
\ref{lem:25:04:6}
with the bound 
\eqref{eq:L4:bounds}, we can 
render the supremum limit (as $h$ tends to $0$) of the right-hand side 
as small as needed by choosing $N$ large enough. 
\end{proof} 
\color{black}

\section{Expansion} 
\label{se:4}

We now complete the proof of Theorem 
\ref{prop:ito:b}. We first prove the result when $X_0 \in H_{\rm sym}^1({\mathbb S}) \cap U^2({\mathbb S})$
with ${\mathbb E} [ \| D X_0 \|_2^{2p}] < \infty$ for any $p \geq 1$. 
{ \ }
\vspace{4pt} 

\textit{First Step.} The very first step in the proof is to expand the difference
$    \varphi ( X^h_{n+1} ) 
    - \varphi ( X^h_n )$. 
Since $\varphi$  is invariant by rearrangement
and $Y_{(n+1)h-}^* = X_{n+1}^h$,
we can rewrite the latter difference as 
$   \varphi ( Y^h_{(n+1)h-} ) 
    - \varphi ( X^h_n )$.
Since $\phi$ is continuous with respect to  the 
$L^2$-norm on ${\mathbb S}$, 
\begin{equation}
\label{prop:25:4:1}
\begin{split}
      \varphi \bigl( X^h_{n+1} \bigr) 
    - \varphi \bigl( X^h_n \bigr)&=  \varphi \bigl( Y^h_{(n+1)h-} \bigr) 
    - \varphi \bigl( X^h_n \bigr)
 = \lim_{N \rightarrow \infty}
    \Bigl[
    \varphi \bigl( Y^{h,N}_{(n+1)h-} \bigr) 
    - \varphi \bigl( Y^{h,N}_{nh} \bigr) \Bigr],
\end{split}
\end{equation}
where we used 
the same notations as in Lemma 
\ref{lem:25:04:6}
in the above right-hand side. We now claim that we can apply It\^o's formula to the increment appearing in the argument of the latter limit. To do so, we must argue that 
$\varphi$ is a smooth function, when restricted to random variables with a finite Fourier expansion. In fact, we first claim that 
$\varphi$ has directional derivatives of order 1 and 2 with respect to $(e_k)_{k \in {\mathbb N}_0}$. We clearly have, for 
any $k \in {\mathbb N}_0$
and
any random variable 
$X \in L^2({\mathbb S},{\mathbb R},\textrm{\rm Leb})$,
\begin{equation*}
\frac{d}{d \epsilon} 
\varphi \bigl( X + \epsilon e_k \bigr)_{\vert 
\epsilon =0} 
= 
    \int_{\mathbb S} \partial_\mu \varphi\bigl({\mathcal L}(X) \bigr)\bigl( X(x) 
    \bigr) e_k(x) dx,
\end{equation*}
and then, 
for any additional $j \in {\mathbb N}_0$,
\begin{equation*}
\begin{split}
\frac{d^2}{d \epsilon_1 d \epsilon_2 } 
\varphi \bigl( X + \epsilon_1 e_k + 
\epsilon_2 e_j \bigr)_{\vert (\epsilon_1,\epsilon_2)=(0,0)} 
&= 
    \int_{\mathbb S} \nabla_y \partial_\mu \varphi\bigl({\mathcal L}(X) \bigr)\bigl( X(x) 
    \bigr) e_k(x)  {e_j(x)}dx
    \\
&\hspace{5pt} +  \int_{\mathbb S}\int_{\mathbb S} \partial^2_\mu  \varphi\bigl({\mathcal L}(X) \bigr)\bigl( X(x),X(y) 
    \bigr) e_k(x) e_j(y) dx dy. 
    \end{split}
\end{equation*}
We deduce that $\varphi$, when restricted to random variables with a finite number of non-zero Fourier coefficients, is twice continuously differentiable and, with intuitive notations, 
\begin{equation*} 
 \begin{split}
\partial_{e_k} \varphi (X)  &=    \int_{\mathbb S}   \partial_\mu \varphi\bigl({\mathcal L}(X) \bigr)\bigl( X(x) 
    \bigr) e_k(x) dx,
    \\
\partial^2_{e_k e_j} \varphi (X)
&=   
    \int_{\mathbb S} \nabla_y \partial_\mu \varphi\bigl({\mathcal L}(X) \bigr)\bigl( X(x) 
    \bigr) e_k(x)  {e_j(x) }dx
    \\
&\hspace{15pt} 
+  \int_{\mathbb S}\int_{\mathbb S} \partial^2_{\mu}\varphi\bigl({\mathcal L}(X) \bigr)\bigl( X(x),X(y) 
    \bigr) e_k(x) e_j(y) dx dy. 
 \end{split}
 \end{equation*}
This says that we can apply It\^o's formula to 
$(\varphi(Y^{h,N}_s))_{nh \leq s \leq (n+1)h}$. 
{By 
\eqref{eq:Yth}, we get, for $s \in [nh,(n+1)h)$,
and with $((\beta_t^k)_{t \geq 0})_{k \in {\mathbb N}_0}$ as in 
\eqref{eq:expansion:W:beta:noises},} 
\begin{equation}
\label{eq:YN:expansion}
\begin{split}
d_s \bigl[ \varphi\bigl(Y^{h,N}_s\bigr)
\bigr]
&= - 4 \pi^2 \sum_{0 \le k \leq N}
k^2  \partial_{e_k} \varphi \bigl( Y_s^{h,N}\bigr)  
\widehat{Y}_s^{k,h} 
ds
+ \sum_{0 \le k \le N}
\lambda_k \partial_{e_k} \varphi \bigl( Y_s^{h,N}\bigr)  d \beta_s^k 
\\
&\hspace{15pt} 
+ \frac12 \sum_{0 \le k \le N}
\lambda_k^2 \partial^2_{e_k e_k} \varphi  (\hat{X}^k)^2\bigl( Y_s^{h,N}\bigr)  ds
\\
&=: d T_1^{h,N}(s) + d T_2^{h,N}(s) + d T_3^{h,N}(s).
\end{split}
\end{equation}
Now,
\begin{equation}
\label{eq:T1N}
\begin{split}
 \frac{d}{ds} T_1^{h,N}(s)  &=
\int_{\mathbb S}   \partial_{\mu}
\varphi\bigl( {\mathcal L}(Y^{h,N}_s) \bigr)
\bigl(Y^N_s(x)\bigr) \Delta Y^{h,N}_s(x) dx
\\
&=-
\int_{\mathbb S} \nabla_y \partial_{\mu}
\varphi\bigl( {\mathcal L}(Y^{h,N}_s) \bigr)
\bigl(Y^{h,N}_s(x)\bigr) \bigl[ D Y^{h,N}_s(x) \bigr]^2 dx.   
\end{split}
\end{equation}
By the same argument, we have
\begin{equation}
\label{eq:T2N}
\begin{split}
d T_2^{h,N}(s) =&
\int_{\mathbb S}  \partial_{\mu}
\varphi\bigl( {\mathcal L}(Y^{h,N}_s) \bigr)
\bigl(Y^{h,N}_s(x)\bigr)  d W^{N}_s(x),   
\end{split}
\end{equation}
with the notation
$W^{N}_t(x) := \sum_{0 \leq k  \leq N} 
\lambda_k
\beta^{k}_t e_k(x)$, $x \in {\mathbb S}$, $t \geq 0$.  
In order to express $d T_3^{h,N}(s)$, we use the same notations as in
\eqref{eq:F1:F2}, but with an additional truncation at level $N$, 
i.e., 
\begin{equation}
\label{eq:FN} 
\begin{split}
&F_1^N(x) := \sum_{0 \leq k  \leq N}  \lambda_k^2 e_k^{{2}}(x), \quad  F_2^N(x,y) := \sum_{0 \leq  k  \leq N}  \lambda_k^2 e_k(x)e_k(y), \quad 
x,y \in {\mathbb S}. 
\end{split}
\end{equation}
We get
\begin{equation}
\label{eq:T3N}
\begin{split}
\frac{d}{ds} T_3^{h,N}(s) &:= \frac12 
    \int_{\mathbb S} \nabla_y \partial_\mu \varphi\bigl({\mathcal L}(Y^{h,N}_s) \bigr)\bigl( Y^{h,N}_s(x) 
    \bigr) F_1^N(x) dx
    \\
&\hspace{15pt} 
+  \frac12   \int_{\mathbb S}\int_{\mathbb S} \partial^2_{\mu} \varphi\bigl({\mathcal L}\bigl(Y^{h,N}_s\bigr) \bigr)\bigl( Y^{h,N}_s(x),Y^{h,N}_s(y) 
    \bigr) F_2^N(x,y) dx dy. 
\end{split}
\end{equation} 

\textit{Second Step.}
We now aim to let $N$ tend to $\infty$ on each $[nh,(n+1)h)$.
The $ds$ terms can be treated as
in  
\eqref{eq:Y:-YN}. The stochastic integral
can be handled by means of It\^o's isometry, using the fact that 
$\sum_{m \in {\mathbb N}_0} 
\lambda_m^2 < \infty$. 
Briefly, we obtain the 
same expansion {as} in 
\eqref{eq:YN:expansion}
but removing (in a very obvious way) the truncation parameter 
$N$ in the various terms and in the subsequent 
displays 
\eqref{eq:T1N}, 
\eqref{eq:T2N},
\eqref{eq:FN}
and
\eqref{eq:T3N}.
Then, using {\eqref{prop:25:4:1}},
we can recombine all 
the steps together. We get, ${\mathbb P}$ almost surely, for all $t \geq 0$, 
(with $F_1$ and $F_2$   as in 
	\eqref{eq:F1:F2}):
\begin{equation*}
\begin{split}
&\varphi\bigl(X_t^h\bigr)
-
\varphi\bigl(X^0\bigr)
= - \int_0^t 
\int_{\mathbb S} \nabla_y \partial_{\mu}
\varphi\bigl( {\mathcal L}(Y^{h}_s) \bigr)
\bigl(Y^{h}_s(x)\bigr) \bigl[ \nabla_x Y^{h}_s(x) \bigr]^2 dx \, ds
\\
&\hspace{15pt} +
\int_0^t 
\int_{\mathbb S}  \partial_{\mu}
\varphi\bigl( {\mathcal L}(Y^{h}_s) \bigr)
\bigl(Y^{h}_s(x)\bigr)  d W_s(x)
  +  {\frac12}
  \int_0^t  \int_{\mathbb S} \nabla_y \partial_\mu \varphi\bigl({\mathcal L}(Y^{h}_s) \bigr)\bigl( Y^{h}_s(x) 
    \bigr) F_1(x) dx \, ds
    \\
&\hspace{15pt} 
+  {\frac12} \int_0^t \int_{\mathbb S}\int_{\mathbb S} \partial^2_{\mu} \varphi\bigl({\mathcal L}\bigl(Y^{h}_s\bigr) \bigr)\bigl( Y^{h}_s(x),Y^{h}_s(y) 
    \bigr) F_2(x,y) dx \, dy \, ds
    \\
    &=: S_1^h(t) + S_2^h(t) + S_3^h(t) + S_4^h(t).
\end{split} 
\end{equation*} 

\textit{Third Step.} We now 
want to let $h$ tend to $0$ in the right-hand side.
The strategy is the same as in the proof of 
Lemma  \ref{lem:25:04:7}. The point is to come back to $(Y^{h,N}_s)_{0 \le s \leq T}$. Notice however that this is different from what is done in the previous step because 
the expansion is now given on the entire interval $[0,t]$. This makes it possible to pass to the limit (in the weak sense) in 
$S_1^h$, $S_3^h$ and $S_4^h$. %
The term $S_2^h$ is somewhat more complicated, but we can easily adapt the works \cite{KurtzProtter1,KurtzProtter2} to pass to the limit therein {(recalling that 
we have the joint convergence of 
$(\tilde X^h_t,W_t)_{t \geq 0}$
and thus of 
$(Y_t^h,W_t)_{t \geq 0}$)}. 
We get the announced expansion, for a fixed $t$, almost surely. 
By continuity, we can easily exchange the quantifiers and get the expansion almost surely, for any $t \geq 0$. 
\vspace{4pt}

It now remains to relax the condition $X_0 \in H^1({\mathbb S})$ {(and just retain the assumption 
${\mathbb E}[ \| X_0 \|^p_2] < \infty$ for all $p \geq 1$)}. In fact, we can apply the expansion to $e^{\varepsilon \Delta} X_0$ for a given $\varepsilon >0$. Indeed, it is obvious that 
$e^{\varepsilon \Delta} X_0  \in H^1({\mathbb S})$ {(and 
${\mathbb E}[ \| D(e^{\varepsilon \Delta} X_0) \|^p_2] < \infty$ for all $p \geq 1$)}. Moreover, it is known (see \cite{delarueHammersley2022rshe})
that 
$e^{\varepsilon \Delta} X_0 \in U^2({\mathbb S})$
if 
$X_0 \in U^2({\mathbb S})$. The point is to let $\varepsilon$ tend to $0$
in the It\^o expansion 
for the process $(X_t^\varepsilon)_{t \geq 0}$ obtained by solving the  rearranged 
SHE with $X_0^\varepsilon$ as initial condition. We already know from 
\cite{delarueHammersley2022rshe} (see for instance the proof of 
Propositions 4.14 and 4.17 therein) that, for any $t \in [0,T]$, there exists a constant 
$C_T>0$ such that, for any $t \in [0,T]$,  
\begin{equation*} 
\sup_{t \in [0,T]} 
\| X_t^{\varepsilon} - X_t \|_2^2
+
\int_0^T 
\| D X_s^{\varepsilon} - D X_s \|_2^2 
ds
\leq C_T \| X_0^\varepsilon - X_0\|_2^2,
\end{equation*}  
which allows one to pass easily  {to the limit $\varepsilon \searrow 0$ in the It\^o expansion. 
\qed

\appendix

\printbibliography  

@incollection {KurtzProtter1,
    AUTHOR = {Kurtz, Thomas G. and Protter, Philip E.},
     TITLE = {Weak convergence of stochastic integrals and differential
              equations},
 BOOKTITLE = {Probabilistic models for nonlinear partial differential
              equations ({M}ontecatini {T}erme, 1995)},
    SERIES = {Lecture Notes in Math.},
    VOLUME = {1627},
     PAGES = {1--41},
 PUBLISHER = {Springer, Berlin},
      YEAR = {1996},
   MRCLASS = {60H05 (60F17)},
  MRNUMBER = {1431298},
MRREVIEWER = {Leszek S\l omi\'{n}ski},
       DOI = {10.1007/BFb0093176},
       URL = {https://doi.org/10.1007/BFb0093176},
}

@incollection {KurtzProtter2,
    AUTHOR = {Kurtz, Thomas G. and Protter, Philip E.},
     TITLE = {Weak convergence of stochastic integrals and differential
              equations. {II}. {I}nfinite-dimensional case},
 BOOKTITLE = {Probabilistic models for nonlinear partial differential
              equations ({M}ontecatini {T}erme, 1995)},
    SERIES = {Lecture Notes in Math.},
    VOLUME = {1627},
     PAGES = {197--285},
 PUBLISHER = {Springer, Berlin},
      YEAR = {1996},
   MRCLASS = {60H05 (60B12 60G57)},
  MRNUMBER = {1431303},
MRREVIEWER = {Leszek S\l omi\'{n}ski},
       DOI = {10.1007/BFb0093181},
       URL = {https://doi.org/10.1007/BFb0093181},
}

@article {Stannat,
    AUTHOR = {Stannat, Wilhelm},
     TITLE = {Long-time behaviour and regularity properties of transition
              semigroups of {F}leming-{V}iot processes},
   JOURNAL = {Probab. Theory Related Fields},
  FJOURNAL = {Probability Theory and Related Fields},
    VOLUME = {122},
      YEAR = {2002},
    NUMBER = {3},
     PAGES = {431--469},
      ISSN = {0178-8051},
   MRCLASS = {60J35 (47D07 60G57 60K35 92D10)},
  MRNUMBER = {1892853},
MRREVIEWER = {Kyle Siegrist},
       DOI = {10.1007/s004400100166},
       URL = {https://doi.org/10.1007/s004400100166},
}

@book {CarmonaDelarueI,
    AUTHOR = {Carmona, Ren\'{e} and Delarue, Fran\c{c}ois},
     TITLE = {Probabilistic theory of mean field games with applications.
              {I}},
    SERIES = {Probability Theory and Stochastic Modelling},
    VOLUME = {83},
      NOTE = {Mean field FBSDEs, control, and games},
 PUBLISHER = {Springer, Cham},
      YEAR = {2018},
     PAGES = {xxv+713},
      ISBN = {978-3-319-56437-1; 978-3-319-58920-6},
   MRCLASS = {60-02 (35R60 49N70 49N90 60H15 60H30 91A15 93E20)},
  MRNUMBER = {3752669},
MRREVIEWER = {Vassili N. Kolokol\cprime tsov},
}

@article {KurtzXiong,
    AUTHOR = {Kurtz, Thomas G. and Xiong, Jie},
     TITLE = {Particle representations for a class of nonlinear {SPDE}s},
   JOURNAL = {Stochastic Process. Appl.},
  FJOURNAL = {Stochastic Processes and their Applications},
    VOLUME = {83},
      YEAR = {1999},
    NUMBER = {1},
     PAGES = {103--126},
      ISSN = {0304-4149},
   MRCLASS = {60H15 (35R60 60K35)},
  MRNUMBER = {1705602},
MRREVIEWER = {Sylvie M\'{e}l\'{e}ard},
       DOI = {10.1016/S0304-4149(99)00024-1},
       URL = {https://doi.org/10.1016/S0304-4149(99)00024-1},
}

@article {KurtzXiong2,
    AUTHOR = {Kurtz, Thomas G. and Xiong, Jie},
     TITLE = {A stochastic evolution equation arising from the fluctuations
              of a class of interacting particle systems},
   JOURNAL = {Commun. Math. Sci.},
  FJOURNAL = {Communications in Mathematical Sciences},
    VOLUME = {2},
      YEAR = {2004},
    NUMBER = {3},
     PAGES = {325--358},
      ISSN = {1539-6746},
   MRCLASS = {60H15 (60B12 60H35 60K35 93E11)},
  MRNUMBER = {2118848},
MRREVIEWER = {Denis Talay},
       URL = {http://projecteuclid.org/euclid.cms/1109868725},
}

@article {Baernstein_correction,
    AUTHOR = {Baernstein, II, Albert},
     TITLE = {Correction to: ``{C}onvolution and rearrangement on the
              circle'' [{C}omplex {V}ariables {T}heory {A}ppl. {\bf 12}
              (1989), no. 1-4, 33--37]},
   JOURNAL = {Complex Variables Theory Appl.},
  FJOURNAL = {Complex Variables. Theory and Application. An International
              Journal},
    VOLUME = {26},
      YEAR = {1995},
    NUMBER = {4},
     PAGES = {381--382},
      ISSN = {0278-1077},
   MRCLASS = {26D15 (26A48 42A85)},
  MRNUMBER = {1315871},
MRREVIEWER = {Peter W. Day},
       DOI = {10.1080/17476939508814799},
       URL = {https://doi.org/10.1080/17476939508814799},
}

@Article{dawsonVaillancourt1995,    
AUTHOR = {Dawson, Donald and Vaillancourt, Jean},
     TITLE = {Stochastic {M}c{K}ean-{V}lasov equations},
   JOURNAL = {NoDEA Nonlinear Differential Equations Appl.},
  FJOURNAL = {NoDEA. Nonlinear Differential Equations and Applications},
    VOLUME = {2},
      YEAR = {1995},
    NUMBER = {2},
     PAGES = {199--229},
      ISSN = {1021-9722},
   MRCLASS = {60H99},
  MRNUMBER = {1328577},
MRREVIEWER = {Marc Brunaud},
       DOI = {10.1007/BF01295311},
       URL = {https://doi.org/10.1007/BF01295311},
}

@article {Konarovsky0,
    AUTHOR = {Konarovskyi, Vitalii},
     TITLE = {A system of coalescing heavy diffusion particles on the real
              line},
   JOURNAL = {Ann. Probab.},
  FJOURNAL = {The Annals of Probability},
    VOLUME = {45},
      YEAR = {2017},
    NUMBER = {5},
     PAGES = {3293--3335},
      ISSN = {0091-1798},
   MRCLASS = {60K35 (60H05 60J55 82B21)},
  MRNUMBER = {3706744},
       DOI = {10.1214/16-AOP1137},
       URL = {https://doi.org/10.1214/16-AOP1137},
}

@article {Konarovskyi,
    AUTHOR = {Konarovskyi, Vitalii},
     TITLE = {On asymptotic behavior of the modified {A}rratia flow},
   JOURNAL = {Electron. J. Probab.},
  FJOURNAL = {Electronic Journal of Probability},
    VOLUME = {22},
      YEAR = {2017},
     PAGES = {Paper No. 19, 31},
   MRCLASS = {60K35 (60D05 60J60 82C22)},
  MRNUMBER = {3622889},
       DOI = {10.1214/17-EJP34},
       URL = {https://doi.org/10.1214/17-EJP34},
}

@article {AndresvRenesse,
    AUTHOR = {Andres, Sebastian and von Renesse, Max-K.},
     TITLE = {Uniqueness and regularity for a system of interacting {B}essel
              processes via the {M}uckenhoupt condition},
   JOURNAL = {Trans. Amer. Math. Soc.},
  FJOURNAL = {Transactions of the American Mathematical Society},
    VOLUME = {364},
      YEAR = {2012},
    NUMBER = {3},
     PAGES = {1413--1426},
      ISSN = {0002-9947},
   MRCLASS = {60J60 (42B37)},
  MRNUMBER = {2869181},
MRREVIEWER = {Ren Ming Song},
       DOI = {10.1090/S0002-9947-2011-05457-7},
       URL = {https://doi.org/10.1090/S0002-9947-2011-05457-7},
}

@incollection {Sturm,
    AUTHOR = {Sturm, Karl-Theodor},
     TITLE = {A monotone approximation to the {W}asserstein diffusion},
 BOOKTITLE = {Singular phenomena and scaling in mathematical models},
     PAGES = {25--48},
 PUBLISHER = {Springer, Cham},
      YEAR = {2014},
   MRCLASS = {60J60 (31C25 47D07)},
  MRNUMBER = {3205035},
MRREVIEWER = {Wilhelm Stannat},
       DOI = {10.1007/978-3-319-00786-1\_2},
       URL = {https://doi.org/10.1007/978-3-319-00786-1_2},
}

@article {vRenesseSturm,
    AUTHOR = {von Renesse, Max-K. and Sturm, Karl-Theodor},
     TITLE = {Entropic measure and {W}asserstein diffusion},
   JOURNAL = {Ann. Probab.},
  FJOURNAL = {The Annals of Probability},
    VOLUME = {37},
      YEAR = {2009},
    NUMBER = {3},
     PAGES = {1114--1191},
      ISSN = {0091-1798},
   MRCLASS = {60G57 (35R60 47D07 58J65 60J60)},
  MRNUMBER = {2537551},
MRREVIEWER = {Ingemar Kaj},
       DOI = {10.1214/08-AOP430},
       URL = {https://doi.org/10.1214/08-AOP430},
}

@article {DawsonMarch,
    AUTHOR = {Dawson, Donald A. and March, Peter},
     TITLE = {Resolvent estimates for {F}leming-{V}iot operators and
              uniqueness of solutions to related martingale problems},
   JOURNAL = {J. Funct. Anal.},
  FJOURNAL = {Journal of Functional Analysis},
    VOLUME = {132},
      YEAR = {1995},
    NUMBER = {2},
     PAGES = {417--472},
      ISSN = {0022-1236},
   MRCLASS = {60J35 (60G57 60J60)},
  MRNUMBER = {1347357},
MRREVIEWER = {S. N. Ethier},
       DOI = {10.1006/jfan.1995.1111},
       URL = {https://doi.org/10.1006/jfan.1995.1111},
}

@book {Dawson,
    AUTHOR = {Dawson, D. A. and Maisonneuve, B. and Spencer, J.},
     TITLE = {\'{E}cole d'\'{E}t\'{e} de {P}robabilit\'{e}s de {S}aint-{F}lour {XXI}---1991},
    SERIES = {Lecture Notes in Mathematics},
    VOLUME = {1541},
      NOTE = {Papers from the school held in Saint-Flour, August
              18--September 4, 1991,
              Edited by P. L. Hennequin},
 PUBLISHER = {Springer-Verlag, Berlin},
      YEAR = {1993},
     PAGES = {viii+352},
      ISBN = {3-540-56622-8},
   MRCLASS = {60-06},
  MRNUMBER = {1242574},
       DOI = {10.1007/BFb0084189},
       URL = {https://doi.org/10.1007/BFb0084189},
}

@article {DoringvRenesse,
    AUTHOR = {D\"{o}ring, Maik and Stannat, Wilhelm},
     TITLE = {The logarithmic {S}obolev inequality for the {W}asserstein
              diffusion},
   JOURNAL = {Probab. Theory Related Fields},
  FJOURNAL = {Probability Theory and Related Fields},
    VOLUME = {145},
      YEAR = {2009},
    NUMBER = {1-2},
     PAGES = {189--209},
      ISSN = {0178-8051},
   MRCLASS = {31C25 (28A33 35P15 47D07 60J35)},
  MRNUMBER = {2520126},
MRREVIEWER = {Andrzej Stos},
       DOI = {10.1007/s00440-008-0166-6},
       URL = {https://doi.org/10.1007/s00440-008-0166-6},
}

@Article{vaillancourt1988,
  author    = {Jean Vaillancourt},
  title     = {On the existence of random {M}c{K}ean–{V}lasov limits for triangular arrays of exchangeable diffusions},
  number    = {4},
  pages     = {431-446},
  volume    = {6},
  journal   = {Stochastic Analysis and Applications},
  publisher = {Taylor & Francis},
  year      = {1988},
}

@Book{daPratoZabczyk2014stochEqnsInfDim,
  author     = {Da Prato, Giuseppe and Zabczyk, Jerzy},
  title      = {Stochastic Equations in Infinite Dimensions},
  doi        = {10.1017/CBO9781107295513},
  edition    = {2},
  publisher  = {Cambridge University Press},
  series     = {Encyclopedia of Mathematics and its Applications},
  collection = {Encyclopedia of Mathematics and its Applications},
  place      = {Cambridge},
  year       = {2014},
}

@Book{baernstein2019symmetrizationInAnalysis,
  author     = {Baernstein II, Albert},
  date       = {2019},
  title      = {Symmetrization in Analysis},
  doi        = {10.1017/9781139020244},
  publisher  = {Cambridge University Press},
  series     = {New Mathematical Monographs},
  collection = {New Mathematical Monographs},
  place      = {Cambridge},
  year       = {2019},
}

@Article{baernstein1989convRearrOnCirc,
  author    = {Baernstein II, Albert},
  title     = {Convolution and rearrangement on circle},
  doi       = {10.1080/17476938908814351},
  eprint    = {https://doi.org/10.1080/17476938908814351},
  number    = {1-4},
  pages     = {33-37},
  volume    = {12},
  journal   = {Complex Variables, Theory and Application: An International Journal},
  publisher = {Taylor & Francis},
  year      = {1989},
}

@Article{delarueHammersley2024rshe_erg,
  author       = {François Delarue and William R. P. Hammersley},
  year         = {v1, v2: 2024},
  journal = {arXiv:2403.16140 (v2)},
  title        = {Rearranged Stochastic Heat Equation: Ergodicity and Related Gradient Descent on the Space of Probability Measures},
}

@Article{delarueHammersley2022rshe,
  author       = {François Delarue and William R. P. Hammersley},
  year         = {2022, v2: 2024},
  journal = {arXiv:2210.01239 (v2)},
  title        = {Rearranged Stochastic Heat Equation},
}

@Article{delarueOuknine2024intrinsicRegMFG,
  author        = {François Delarue and Youssef Ouknine},
  title         = {Intrinsic regularization by noise for $1d$ mean field games},
  year          = {2024},
journal      = {arXiv: 2401.13844},
  %archiveprefix = {arXiv},
  primaryclass  = {math.PR},
}
\end{document}